\documentclass[12pt]{amsart}
\usepackage[margin=1.15in]{geometry}

\usepackage{amsmath, amscd, amssymb}
\usepackage{mathrsfs}
\usepackage{amsfonts}
\usepackage{latexsym,epsfig}
\usepackage{amsmath, amscd, amsthm}
\usepackage[pdftex]{color}
\usepackage[colorlinks,citecolor=blue,linkcolor=red]{hyperref}
\usepackage{overpic}

\newtheorem{thm}{Theorem}[section]

\newtheorem{lemma}[thm]{Lemma}
\newtheorem{prop}[thm]{Proposition}
\newtheorem{ques}[thm]{Question}
\newtheorem{conj}[thm]{Conjecture}
\theoremstyle{definition}

\newtheorem{example}[thm]{Example}

\newtheorem{claim}[thm]{Claim}

\newcommand{\Z}{\mathbb{Z}}

\newcommand{\N}{\mathbb{N}}

\DeclareMathOperator{\lcm}{lcm}

\newcommand{\la}{\langle}
\newcommand{\ra}{\rangle}

\newcommand{\p}{\partial}

\definecolor{amethyst}{rgb}{0.6, 0.4, 0.8}

\newcommand{\hide}[1]{}

\title{Hyperbolic groups that are not commensurably coHopfian}
\author{Emily Stark, Daniel Woodhouse }
\date{ \today }
\address{Department of Mathematics\\
Technion - Israel Institute of Technology \\
Haifa 32000 \\
Israel }
\email{emily.stark@campus.technion.ac.il, woodhouse.da@campus.technion.ac.il}

\thanks{The first author was supported by the Azrieli Foundation and was supported in part at the Technion by a Zuckerman Fellowship. The second author was supported by the Israel Science Foundation (grant 1026/15).}

\begin{document}
 
\begin{abstract}
  Sela proved every torsion-free one-ended hyperbolic group is coHopfian. We prove there exist torsion-free one-ended hyperbolic groups that are not commensurably coHopfian. In particular, we show that the fundamental group of every simple surface amalgam is not commensurably coHopfian. 
\end{abstract}

\maketitle

\section{Introduction}

  Hyperbolic groups, introduced by Gromov in his seminal essay~\cite{gromov}, form a broad family of finitely presented groups satisfying a coarse notion of negative curvature.
  The large scale geometry reflected in the definition has a variety of algebraic implications; for example, hyperbolic groups satisfy the Tits alternative and each of Dehn's decision problems -- the word, conjugacy, and isomorphism problems~\cite{selaIsomorphism, DahmaniGuirardel} -- are solvable. 
  A major development in the understanding of (torsion free) hyperbolic groups was the body of work developed by Rips and Sela, each independently and in partnership.
  The techniques they developed involved studying the action of hyperbolic groups on $\mathbb{R}$-trees and then developing a JSJ decomposition -- a canonical splitting of a hyperbolic group along two-ended subgroups.

    A group is {\it coHopfian} if it is not isomorphic to any of its proper subgroups. Sela~\cite{sela97}, building on work of Rips--Sela~\cite{RipsSela94}, proved a torsion-free hyperbolic group is coHopfian if and only if it is freely indecomposable. Thus, by the ``easy'' direction of Stallings' Theorem~\cite{stallings}, a torsion-free one-ended hyperbolic group is coHopfian.
    In his thesis, Moioli~\cite[Theorem 1.0.8]{Moioli} generalized this statement to prove every one-ended hyperbolic group is coHopfian.
    A group $\Gamma$ is {\it commensurably coHopfian} if no finite-index subgroup of $\Gamma$ is isomorphic to an infinite-index subgroup of $\Gamma$. 
    As we explain below, this variation of coHopficity has its own merits as, unlike the classical notion, it is a commensurability invariant.
    Strebel~\cite{Strebel77} proved that the infinite-index subgroups of a Poincar\'{e} duality group have strictly smaller cohomological dimension than the group. Thus, Poincar\'{e} duality groups are commensurably coHopfian. 
    In particular, fundamental groups of closed hyperbolic manifolds are commensurably coHopfian.
    Following the result of Sela, it was natural to ask if one ended hyperbolic groups are commensurably coHopfian.

    In this paper, we exhibit one-ended hyperbolic groups that are not commensurably coHopfian, answering a question of Whyte on Bestvina's Problem list~\cite[(Whyte, Q. 1.12)]{bestvinaprob} and also asked by Kapovich~\cite[Section 5]{kapovich12}.
    
    \begin{thm} \label{thm:main_thm}
      There exist one-ended hyperbolic groups that are not commensurably coHopfian.
    \end{thm}
    
    Our proof of Theorem~\ref{thm:main_thm} is topological. 
    A {\it simple surface amalgam} is the union of a finite collection of at least three surfaces with negative Euler characteristic and precisely one boundary component each and which have their boundary components identified. These spaces have been studied in~\cite{malone, stark, danistarkthomas, starkwoodhouse}.
    Such spaces are one-ended and Gromov hyperbolic.
    Hyperbolicity is an immediate consequence of the Bestvina-Feighn combination theorem~\cite{bestvinafeighn92}, since free groups are amalgamated along malnormal subgroups.
    Moreover, by assigning each surface a suitable Reimannian metric, every simple surface amalgam admits a metric with a CAT(-1) universal cover (see~\cite[Theorem~II.5.4]{bridsonhaefliger} or Section 4 in~\cite{stark}).
    We exhibit in Section~\ref{sec:ex} a simple surface amalgam $X$ and two finite covers $X_1 \rightarrow X$ and $X_2 \rightarrow X$ so that the space $X_1$ $\pi_1$-injectively embeds in the space $X_2$ and such that $\pi_1 (X_1)$ embeds as an infinite-index subgroup of $\pi_1(X_2)$. 
    See Figure~\ref{figure:coverTrick}. 
     The construction given in Section~\ref{sec:ex} does not immediately extend to simple surface amalgams in which the subsurfaces have different Euler characteristics. Nonetheless, we prove the following in Section~\ref{sec:ssa}.
    
    \begin{thm}
     The fundamental group of any simple surface amalgam is not commensurably coHopfian. 
    \end{thm}

    Bowditch~\cite{bowditch} proved that if $G$ is a one-ended hyperbolic group that is not Fuchsian, then there is a canonical graph of groups decomposition of $G$, called the JSJ decomposition of $G$, with edge groups that are two-ended and vertex groups of three types: two-ended; maximally hanging Fuchsian; and quasi-convex rigid vertex groups not of the first two types. For background, see \cite{scottwall,serre,guirardellevitt}. 
    We conjecture that for a one-ended hyperbolic group the commensurably coHopfian property is related to the existence of maximal hanging Fuchsian vertex groups in the JSJ decomposition of the group over two-ended subgroups. 
    
    \begin{conj} \label{conj}
     Let $\Gamma$ be a one-ended hyperbolic group that is not Fuchsian. If $\Gamma$ is not commensurably coHopfian, then its JSJ decomposition contains a maximal hanging Fuchsian vertex group. Moreover, if the JSJ decomposition of $\Gamma$ only contains maximal hanging Fuchsian vertex groups and 2-ended vertex groups, then $\Gamma$ is not commensurably coHopfian.
    \end{conj}
    
    In this paper, the embeddings constructed are quasi-isometric embeddings and are surely not representative. Indeed, highly distorted subgroups may be counterexamples to Conjecture~\ref{conj}, so a quasi-convexity assumption may be required. 
        
    In Section~\ref{sec:mixed} we present two examples of one-ended hyperbolic groups whose JSJ decomposition contains both maximal hanging Fuchsian and rigid vertex groups and so that one group is commensurably coHopfian and the other is not. We summarize our examples and open problems in Section~\ref{sec:summary}. 
  
  \subsection*{Quasi-isometrically coHopfian}
  The quasi-isometrically coHopfian condition is a related coarse notion for metric spaces.
  A metric space $X$ is {\it quasi-isometrically coHopfian} if every quasi-isometric embedding of $X$ into itself is a quasi-isometry. 
  This property holds for certain coarse PD(n) spaces~\cite{KapovichKleiner05}\cite[Section 3]{KapovichLukyanenko12} and has been studied for certain Gromov hyperbolic spaces~\cite{Merenkov10} and non-uniform lattices in rank-one semisimple Lie groups~\cite{KapovichLukyanenko12}.
  The infinite-index embeddings we give in Section~\ref{sec:ex} and Theorem~\ref{thm:ssa} are retractions and therefore are quasi-isometric embeddings that are not quasi-isometries.
  These are the only known examples, as far as we know, of one-ended hyperbolic groups that are not quasi-isometrically coHopfian.
  
  \subsection*{Motivation of the terminology}
  
  Kapovich~\cite{kapovich12} uses the term \emph{weakly} coHopfian instead of commensurably coHopfian.  We abandon this terminology, since the property is not weaker than the coHopfian property: every one-ended hyperbolic group is coHopfian by the theorem of Sela, but not every one-ended hyperbolic group is commensurably coHopfian as shown here. However, in general, the commensurably coHopfian property defined in this paper is not a stronger condition than the coHopfian property. For example, the integers $\Z$ are commensurably coHopfian, but not coHopfian. 
  The adjective {\it commensurably} is justified by the fact that being commensurably coHopfian is an abstract commensurability invariant, which follows from the lemma below. The coHopfian property, on the other hand, is not an abstract commensurability invariant by work of Cornulier~\cite[Appendix A]{Cornulier16}.
  
  \begin{lemma}
   If $H \leq G$ is a finite-index subgroup, then $H$ is commensurably coHopfian if and only if $G$ is commensurably coHopfian.
  \end{lemma}

  \begin{proof}
    If $G$ is commensurably coHopfian, then $H$ is commensurably coHopfian. Indeed, otherwise, there exists a finite-index subgroup $H' \leq H\leq G$ with an infinite-index embedding $\varphi: H' \rightarrow H \leq G$, contradicting the commensurably coHopficity of $G$.
  
     Conversely, suppose $G$ is not commensurably coHopfian. Then, there exists a finite-index subgroup $G' \leq G$ and an embedding $\varphi: G' \rightarrow G$ so that $\varphi(G')$ is an infinite-index subgroup of $G$. The intersection $\varphi^{-1}(H) \cap H$ is a finite-index subgroup of $H$, since the preimage $\varphi^{-1}(H)$ is a finite-index subgroup of $G$. Therefore $\varphi$ restricts to an infinite-index embedding $\varphi^{-1}(H) \cap H \rightarrow H$. Thus, $H$ is not commensurably coHopfian.
  \end{proof}

  The following example, explained to the authors by Cornulier, shows that the commensurably coHopfian property is not a quasi-isometry invariant.
  
  \begin{example} \label{ex:notQIinv} (Cornulier)
	  Let $\Gamma$ be an arithmetic lattice in $\textrm{SL}_2(\mathbb{Q}_p) \times \textrm{SL}_2(\mathbb{Q}_q)$ for suitable primes $p, q$.
   The group $\Gamma$ acts geometrically on the product of the associated Bruhat-Tits buildings: the product of two trees.
   Thus, the group $\Gamma$ is quasi-isometric to $\mathbb{F}_2 \times \mathbb{F}_2$, the product of two free groups, which is not commensurably coHopfian.
   
    The commensurable coHopficity of $\Gamma$ can be deduced from Margulis' Superrigidity Theorem.
    Indeed, suppose $\Gamma' \leqslant \Gamma$ is a finite-index subgroup and $\phi: \Gamma' \rightarrow \Gamma$ is an infinite-index embedding.
	  Then Superrigidity (see ~\cite[Prop. VII.5.3, p225]{Margulis}, or alternatively \cite[Appendix C]{WitteMorris}) implies that $\phi$ extends to a continuous homomorphism $\Phi: \textrm{SL}_2(\mathbb{Q}_p) \times \textrm{SL}_2(\mathbb{Q}_q) \rightarrow \textrm{SL}_2(\mathbb{Q}_p) \times \textrm{SL}_2(\mathbb{Q}_q)$.
    It then follows from the representation theory of $p$-adic Lie groups that $\Phi$ is an automorphism, contradicting the fact that $\phi$ embeds $\Gamma'$ as an infinite-index subgroup.

  \end{example}

  
  We also note that the following question of Bestvina remains open. 
  
  \begin{ques}
   Does there exist a one-ended hyperbolic group that contains isomorphic finite-index subgroups of different index? 
  \end{ques}

  \subsection{Summary} \label{sec:summary}
    
  The table below summarizes the results in this paper and related open problems. 
  
  \vskip.2in
  
  \Small{
  \begin{center}
   \begin{tabular}{|c||c|c|c||c|}
\hline   
                    & JSJ decomposition	& JSJ decomposition    	& JSJ decomposition		&  $G$ is one-ended   	\\
                    & has only 2-ended	&  has 2-ended,      	& has only 2-ended		&  and hyperbolic\\
$G$ hyperbolic		& and maximal	 	&  maximal hanging  	& and rigid 			&  with trivial JSJ\\
                    & hanging Fuchsian 	& Fuchsian, and rigid	& vertex groups 		& 	decomposition\\
                    & vertex groups  	&  vertex groups	    &				        &	\\
\hline
		&			&			&				&	\\
Commensurably 	& 	Open		& Example~\ref{ex:sc}	&  Example~\ref{prop:rsc}	& Poincar\'{e}	\\	
coHopfian	&	Problem		&			&				& duality groups	\\
		&			&			&				& \cite{Strebel77}	\\
\hline
		&			&			&				&	\\
Not commensurably	& Section~\ref{sec:ex};	& Example~\ref{ex:nsc}	&		Open		&  Open	\\	
coHopfian	& Theorem~\ref{thm:ssa}	&			&		Problem		& Problem	\\
		&			&			&				&	\\
\hline		
   \end{tabular}
   \vskip.1in
   
  \end{center}
}
 
\normalsize

Section~\ref{sec:ex} demonstrates the main example of the paper, and Section~\ref{sec:ssa} generalizes the main example to prove that all simple surface amalgams have fundamental groups that are not commensurably coHopfian. Section~\ref{sec:mixed} gives examples of one-ended hyperbolic groups with mixed JSJ decompositions such that some groups are commensurably coHopfian and others are not. Section~\ref{sec:mixed2} gives an example of commensurably coHopfian one-ended hyperbolic groups that have non-trivial JSJ decomposition and vertex groups that are only two-ended and rigid.

  \subsection*{Acknowledgments}

The authors thank Henry Wilton for bringing this question to our attention. The authors are grateful for helpful discussions with Daniel Groves regarding Example~\ref{ex:sc}, and with Genevieve Walsh. We thank Ilya Kapovich for pointing out the quasi-isometrically coHopfian condition. We thank Yves de Cornulier for explaining Example~\ref{ex:notQIinv}. We thank the referees for their invaluble comments.

  \section{The main example} \label{sec:ex}
 
   A \emph{simple surface amalgam} $X$ is the union of a finite set of surfaces $\Sigma_1, \ldots, \Sigma_k$ with $k \geq 3$ and $\chi(\Sigma_i) < 0$ such that $\p \Sigma_i \cong S^1$ and all boundary components are identified to a single copy of the circle $S^1$ by a homeomorphism. 
   
   The following lemma determines the finite covers of a surface with boundary.
   
   \begin{lemma} \cite[Lemma 3.2]{neumann01} \label{lem:surfaceCovers}
    Let $\Sigma$ be an oriented surface with positive genus.
    Fix a positive integer $d$.
    For each boundary component of $\Sigma$, pick a collection of degrees summing to $d$.
    Then a $d$-sheeted covering $\Sigma' \rightarrow \Sigma$ exists with the prescribed degree coverings in the preimage of each boundary component of $\Sigma$ if and only if the total number of boundary components of $\Sigma'$ has the same parity as $d \chi(\Sigma)$.
   \end{lemma}
   
   We will repeatedly use Lemma~\ref{lem:surfaceCovers} to construct finite covers of a surface amalgam.
   Indeed, given a simple suface amalgam $X$ constructed from surfaces $\Sigma_1, \ldots, \Sigma_k$, we can specify covers of each $\Sigma_i$ and then glue them together along boundary components provided the covering degrees of identified boundary components over the amalgamating curve in $X$ match.
   
    \begin{figure}
	\begin{overpic}[width=.65\textwidth,tics=10, ]{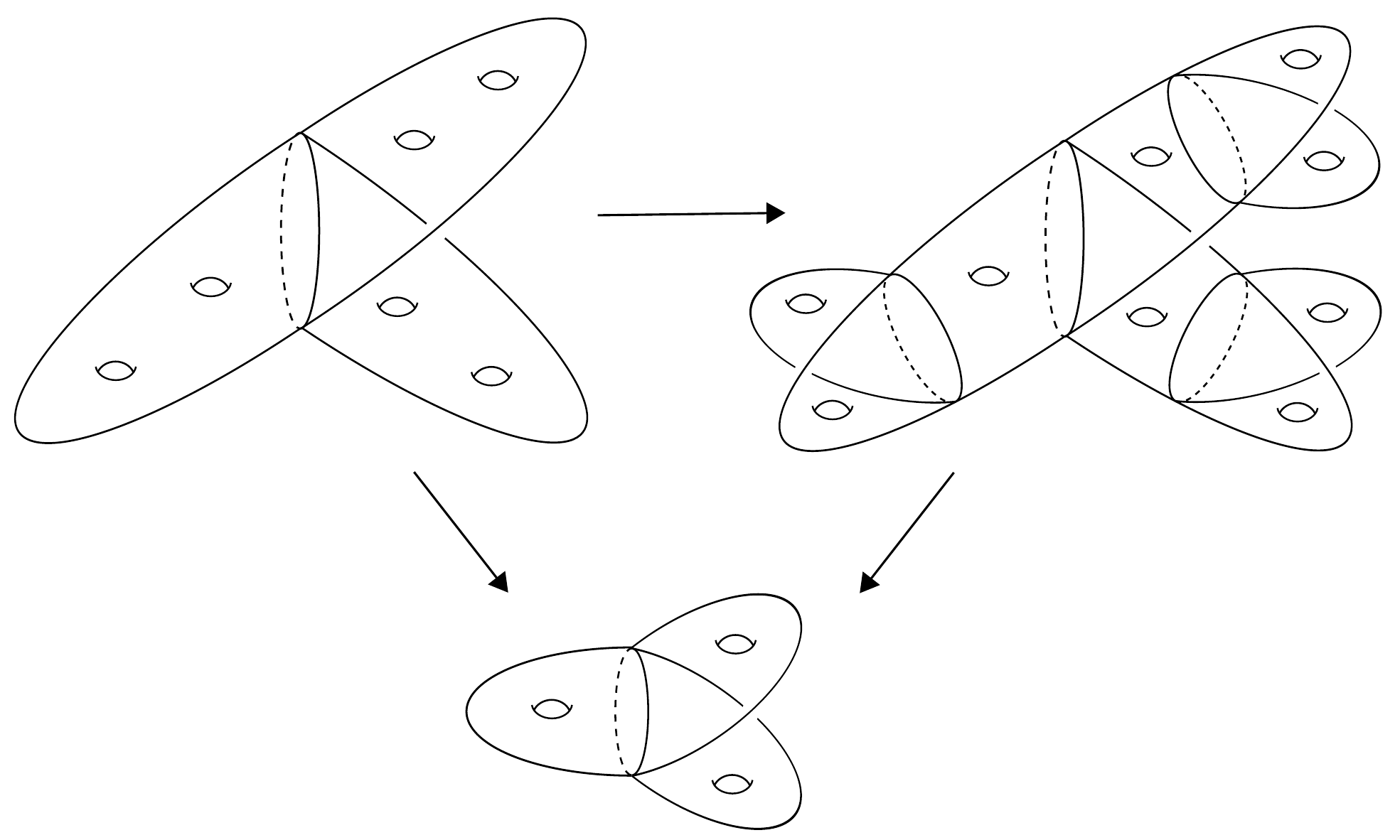} 
          \put(47,47){$\phi$}
          \put(27,8){$X$}
          \put(13,50){$X_1$}
          \put(66,50){$X_2$}
          \put(29,18){$f_1$}
          \put(64,18){$f_2$}
	\end{overpic}
	\caption{\small{  Two finite covers of a simple surface amalgam $X$. The space $X_1$ embeds as a retract of the space $X_2$, such that $\pi_1 (X_1)$ embeds as an infinite-index subgroup of $\pi_1 (X_2)$. Thus, the group $\pi_1(X)$ contains a finite-index subgroup isomorphic to $\pi_1(X_1)$ and an infinite-index subgroup isomorphic to $\pi_1(X_1)$.}}
	\label{figure:coverTrick}
    \end{figure}
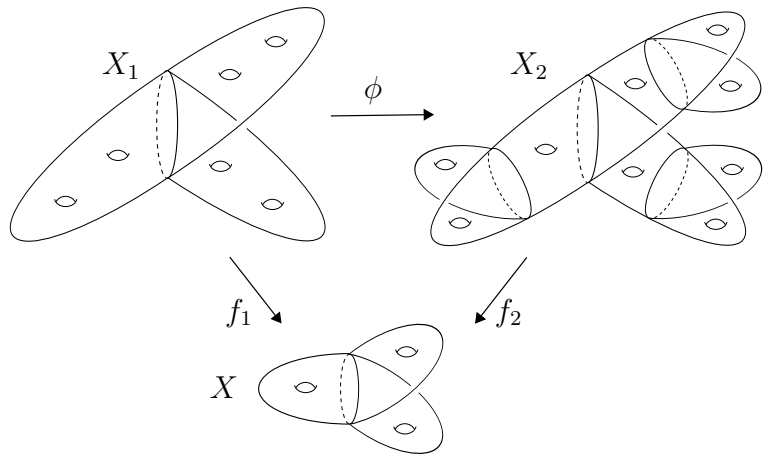
   
   \begin{proof}[Proof of Theorem~\ref{thm:main_thm}]
     Let $X$ be a simple surface amalgam with subsurfaces $\Sigma_1, \Sigma_2, \Sigma_3$, where $\Sigma_i$ is a surface of genus one with a single boundary component. Demonstrating that $\pi_1(X)$ is not commensurably coHopfian follows from considering Figure~\ref{figure:coverTrick}.
   
    We first construct a degree-3 cover $f_1: X_1 \rightarrow X$.
    By Lemma~\ref{lem:surfaceCovers}, there exists a degree-3 cover $\Sigma_i' \rightarrow \Sigma_i$ so that $\Sigma_i'$ has a single boundary component for $i \in \{1,2,3\}$. By an elementary Euler characteristic computation, the surface $\Sigma_i'$ has genus two.
    The boundary components of each $\Sigma_i'$ for $i \in \{1,2,3\}$ can be identified to each other by a homeomorphism to construct a $3$-sheeted cover $f_1: X_1 \rightarrow X$.
    
    We now build a degree-4 cover $f_2: X_2 \rightarrow X$.
    By Lemma~\ref{lem:surfaceCovers}, there exists a degree-2 cover $\Sigma_i'' \rightarrow \Sigma_i$ so that $\Sigma_i''$ has two boundary components for $i \in \{1,2,3\}$. Again, by an elementary Euler characteristic computation, the surface $\Sigma_i''$ has genus one. By identifying a single boundary component from each $\Sigma_i''$ and attaching copies of $\Sigma_j, \Sigma_k$ to the other boundary component of $\Sigma_i''$ we obtain the $4$-sheeted covering $f_2: X_2 \rightarrow X$.
        
    There is a $\pi_1$-injective proper embedding $\phi: X_1 \rightarrow X_2$ as shown in Figure~\ref{figure:coverTrick} that yields an embedding of $\pi_1(X_1)$ in $\pi_1(X_2)$ as an infinite-index subgroup. Therefore, $\pi_1(X)$ contains a finite-index subgroup that is isomorphic to an infinite-index subgroup.
   \end{proof}

 \section{Simple surface amalgams are not commensurably coHopfian} \label{sec:ssa}
 
  \begin{thm} \label{thm:ssa}
    If $G$ is the fundamental group of a simple surface amalgam, then $G$ is not commensurably coHopfian. 
  \end{thm}
  \begin{proof}
    Let $G$ be the fundamental group of a simple surface amalgam $X$ with $k$ subsurfaces $\Sigma_1, \ldots, \Sigma_k$.
    We construct a degree-$2$ cover $\widehat{X} \rightarrow X$ by an application of Lemma~\ref{lem:surfaceCovers}. 
    (This step allows us to resolve any parity issues in the future application of Lemma~\ref{lem:surfaceCovers}.)
    Let $\widehat{X}$ be the union of $k$ surfaces $\widehat{\Sigma}_1, \ldots, \widehat{\Sigma}_k$, where $\chi(\widehat{\Sigma}_i) = 2\chi(\Sigma_i)$.
    Let $\widehat{\chi}_i = \chi(\widehat{\Sigma}_i)$.
    The surface $\widehat{\Sigma}_i$ has two boundary components $\gamma_i$ and $\gamma_i'$, and $\widehat{X}$ is obtained by identifying the curves $\{\gamma_i \, | \, 1 \leq i \leq k\}$ to a single curve $\gamma$ and the curves $\{\gamma_i' \, | \, 1 \leq i \leq k\}$ to a single curve $\gamma'$.  
    
    As in Section~\ref{sec:ex}, we will construct two finite covers $X'$ and $X''$ of the space $\widehat{X}$ so that the space $X'$ embeds $\pi_1$-injectively in the space $X''$, such that $\pi_1 (X')$ embeds as an infinite-index subgroup of $\pi_1 (X'')$. See Figure~\ref{figure:surfacesss}. This construction relies on the next claim. 
    
        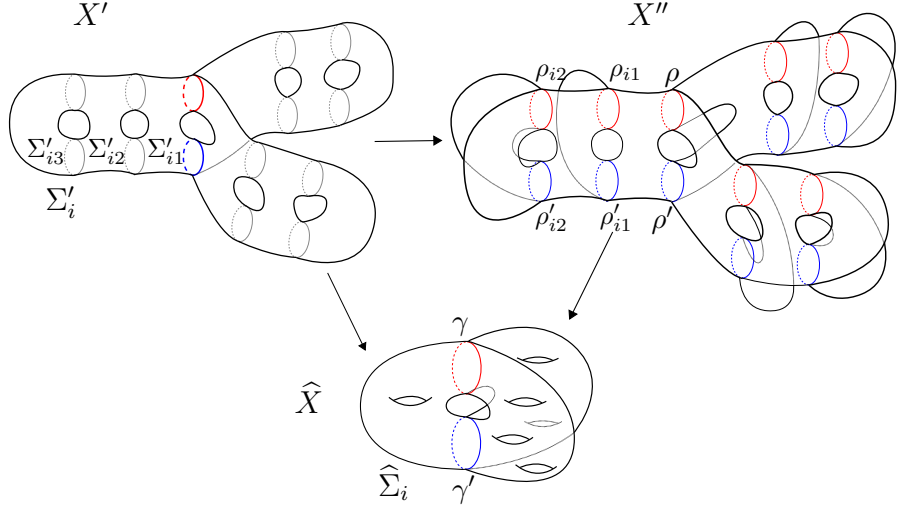
\begin{figure}
	\begin{overpic}[width=.78\textwidth,tics=3 ]{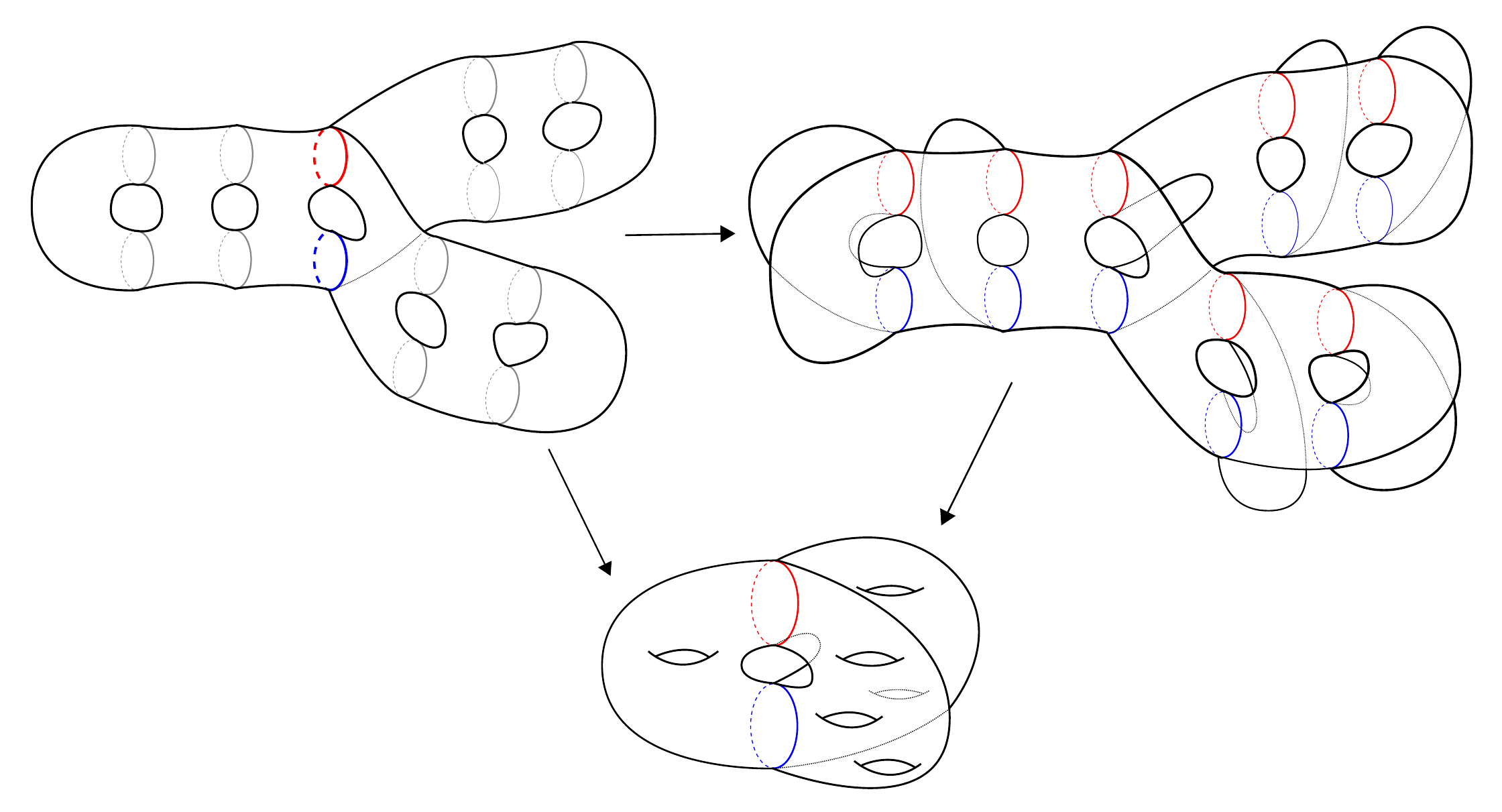} 
	\put(9,51){$X'$}
	\put(33,9){$\widehat{X}$}
	\put(69,51){$X''$}
	\put(6,31){$\Sigma_i'$}
	\put(42,.06){$\widehat{\Sigma}_i$}
	\put(50,18){\small{$\gamma$}}
	\put(50,-0.2){\small{$\gamma'$}}
	\put(73,45){\small{$\rho$}}
	\put(71.6,29){\small{$\rho'$}}
	\put(67,45.5){\Small{$\rho_{i1}$}}
	\put(66,29.5){\Small{$\rho_{i1}'$}}
	\put(59,45.5){\Small{$\rho_{i2}$}}
	\put(59,29.5){\Small{$\rho_{i2}'$}}
	\put(4,36.7){\Small{$\Sigma_{i3}'$}}
	\put(10.6,36.7){\Small{$\Sigma_{i2}'$}}
	\put(17,36.7){\Small{$\Sigma_{i1}'$}}
	\end{overpic}
	\caption{\small{The spaces $X'$ and $X''$ finitely cover the space $\widehat{X}$, and the space $X'$ embeds $\pi_1$-injectively in the space $X''$, such that $\pi_1(X')$ embeds as an infinite-index subgroup of $\pi_1 (X'')$. Thus, the fundamental group of $\widehat{X}$ is not commensurably coHopfian since it contains a finite-index subgroup and an infinite-index subgroup isomorphic to $\pi_1(X')$.}}
	\label{figure:surfacesss}
    \end{figure}
    
     \begin{claim} \label{claim:ints}
       There exists a set of positive integers $\{D,d, d_i \, | \, 1 \leq i \leq k\}$ so that 
    \begin{eqnarray*}
     (d + d_1) \widehat{\chi}_1 \; + \; 2d_1 \widehat{\chi}_k \; + \;  d_1 \widehat{\chi}_1 & = & D\widehat{\chi}_1 \\
     (d + d_2) \widehat{\chi}_2 \; + \; 2d_2 \widehat{\chi}_1 \; + \;  d_2 \widehat{\chi}_2 & = & D\widehat{\chi}_2 \\
      & \vdots &  \\
     (d + d_k) \widehat{\chi}_k + 2d_k \widehat{\chi}_{k-1} + d_k \widehat{\chi}_k & = & D\widehat{\chi}_k.
    \end{eqnarray*}
        \end{claim}
        
         \begin{proof}[Proof of Claim.]
	Rewrite the $i$-th equation in the following form (with indices mod $k$):
	\[
	 2 d_i (\widehat{\chi}_{i-1}+ \widehat{\chi}_{i}) = (D - d)\widehat{\chi}_i.
	\]
        Choose positive integers $D$ and $d$ such that $D > d$ and $(D-d)$ is divisible by $2 \lcm\{(|\widehat{\chi}_{i-1} + \widehat{\chi}_{i}|) \, \mid \, 1 \leq i \leq k\}$. Then, we obtain positive integers 
        \[
         d_i := \frac{(D - d)\widehat{\chi}_i}{2 (\widehat{\chi}_{i-1} + \widehat{\chi}_{i})} \geq 1.
        \]
      \end{proof}
    
    Let $D, d, d_1, \ldots, d_k$ be positive integers satisfying the equations given by Claim~\ref{claim:ints}.
    There exists a degree-$D$ cover $X' \rightarrow \widehat{X}$ constructed as follows. Let $X'$ be the union of $k$ surfaces $\Sigma_1', \ldots, \Sigma_k'$, so that $\chi(\Sigma_i') = D \widehat{\chi}_i = 2D\chi_i$, the surface $\Sigma_i'$ has two boundary components $\rho_i$ and $\rho_i'$, and ${X}'$ is obtained by identifying the curves $\{\rho_i \, | \, 1 \leq i \leq k\}$ to a single curve $\rho$ and the curves $\{\rho_i' \, | \, 1 \leq i \leq k\}$ to a single curve $\rho'$. 
    By Lemma~\ref{lem:surfaceCovers}, there exists a degree-$D$ covering map $\Sigma_i' \rightarrow \widehat{\Sigma}_i$ which restricts to a degree-$D$ cover on each of the boundary components $\rho \rightarrow \gamma$ and $\rho'\rightarrow \gamma'$. Thus, these maps glue to yield a degree-$D$ cover $X' \rightarrow \widehat{X}$.
    
    To build the space $X''$, for $i \in \{1, \ldots, k\}$ we will partition each surface $\Sigma_i' \subset X'$ into three subsurfaces, $\Sigma_{i1}', \Sigma_{i2}', \Sigma_{i3}'$ and attach additional subsurfaces to the boundary curves of $\Sigma_{i2}'$ as follows.
    The construction is illustrated in Figure~\ref{figure:surfacesss}. 
    In particular, the construction ensures that the space $X'$  embeds in $X''$ $\pi_1$-injectively, and $\pi_1 (X')$ embeds as an infinite-index subgroup of $\pi_1 (X'')$.
    Let $\Sigma_{i1}'$ be the subsurface of $\Sigma_i'$ with Euler characteristic $(d+d_i)\widehat{\chi}_i$ and four boundary components, two of which are the curves $\rho_i$ and $\rho_i'$; call the other boundary curves $\rho_{i1}$ and $\rho_{i1}'$. 
    Let $\Sigma_{i2}'$ be the subsurface with Euler characteristic $2d_i\widehat{\chi}_{i-1}$ (subscript mod $k$) and four boundary components, two of which are $\rho_{i1}$ and $\rho_{i1}'$; call the other boundary curves $\rho_{i2}$ and $\rho_{i2}'$.
    Finally, let $\Sigma_{i3}'$ be the subsurface with Euler characteristic $d_i\widehat{\chi}_i$ and two boundary curves, $\rho_{i2}$ and $\rho_{i2}'$. 
    Claim~\ref{claim:ints} implies that we indeed have the decomposition $\Sigma_i' \cong \Sigma_{i1}' \cup \Sigma_{i2}' \cup \Sigma_{i3}'$, since the Euler characteristics  of  $\Sigma_{i1}'$,  $\Sigma_{i2}'$, and $\Sigma_{i3}'$ sum to $D\widehat{\chi}_i$ .     
    For $i \in \{1, \ldots, k\}$ attach $k-2$ surfaces $\{\Sigma^i_j \, | \, j\in \{1, \ldots, k\}, j \neq i, i-1\}$ with two boundary components and Euler characteristics $\chi(\Sigma^i_j) = d_i\widehat{\chi}_j$ to the curves $\{\rho_{i1},\rho_{i1}'\}$.     
    Similarly, attach $k-2$ surfaces $\{\Sigma'^i_j \, | \, j\in \{1, \ldots, k\}, j \neq i, i-1\}$ with two boundary components and Euler characteristics $\chi(\Sigma^i_j) = d_i\widehat{\chi}_j$ to the pair of curves $\{\rho_{i2},\rho_{i2}'\}$. 
    
    We now prove there exists a degree-$(d + 2 \sum_{i=1}^k d_i)$ covering map  $X'' \rightarrow \widehat{X}$. We describe the cover on the branching curves of $X''$, and then we use Lemma~\ref{lem:surfaceCovers} to show the cover extends to all of $X''$. 
    As above, suppose the curves $\{\rho_i\}_{i=1}^k$ and $\{\rho_i'\}_{i=1}^k$ are glued together to form the curves $\rho$ and $\rho'$ in $X''$, respectively. Then, $\rho$ and $\rho'$ cover the curves $\gamma$ and $\gamma'$ by degree $d$. 
    For all $i \in \{1, \ldots, k\}$, the curves $\rho_{i1}$ and $\rho_{i2}$ cover the curve $\gamma$ by degree $d_i$, and the curves $\rho_{i1}'$ and $\rho_{i2}'$ cover the curve $\gamma'$ by degree $d_i$. 
    By Lemma~\ref{lem:surfaceCovers}, there exists a degree $(d+d_i)$ cover $\Sigma'_{i1} \rightarrow \Sigma_i'$, a degree $2d_i$ cover $\Sigma'_{i2} \rightarrow \Sigma_{i-1}'$, and a degree $d_i$ cover $\Sigma_{i3}' \rightarrow \Sigma_i'$.
    By Lemma~\ref{lem:surfaceCovers}, there are degree $d_i$ covers $\Sigma^i_j \rightarrow \Sigma_j'$ and $\Sigma'^i_j \rightarrow \Sigma_j'$. 
    Since these covering maps agree on their intersection, there exists a finite cover $X'' \rightarrow \widehat{X}$. 
  \end{proof} 

  \section{Examples with mixed JSJ decomposition} \label{sec:mixed}

    \begin{figure}        
	\begin{overpic}[width=.8\textwidth, tics=5 ]{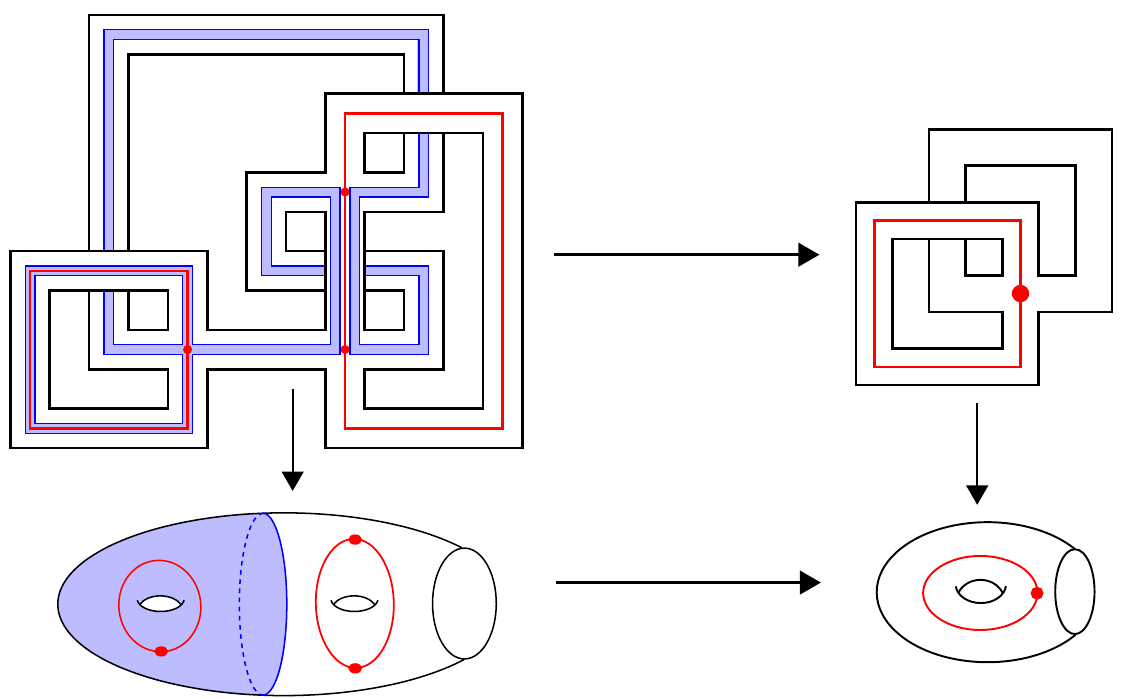} 
	\put(60,41.5){3}
	\put(60,12.5){3}
	\put(22,22){$\cong$}
	\put(83,21){$\cong$}
	\end{overpic}
	\caption{\small{ A degree-3 cover of a surface $\Sigma$ of genus one with one boundary component by a surface of genus two and one boundary component. The red curve on $\Sigma$ has two pre-images that lie in subsurfaces separated by the blue curve. The existence of such a cover is evident from the ``fat-graph'' representation of the surfaces, drawn on the top row. The vertical homeomorphisms are color-preserving. }}
	\label{figure:mixedone}
    \end{figure}
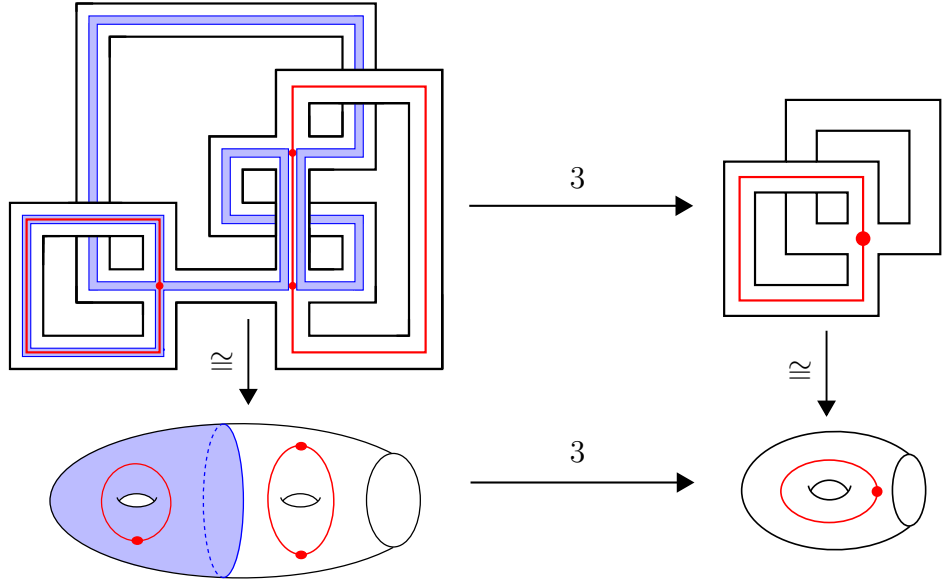
 
  \begin{example} \label{ex:nsc}
   (Not commensurably coHopfian.) We adapt the proof in Section~\ref{sec:ex} to exhibit a one-ended hyperbolic group $G$ whose JSJ decomposition contains both maximal hanging Fuchsian vertex groups and rigid vertex groups and so that $G$ is not commensurably coHopfian. An illustration of this example appears in Figure~\ref{figure:mixedone} and Figure~\ref{figure:mixedthree}. Let $X_0$ be a simple surface amalgam with subsurfaces $\Sigma_1, \Sigma_2, \Sigma_3$, where $\Sigma_i$ is a surface of genus one with a single boundary component. Let $a_i$ be an essential simple closed curve on $\Sigma_i$ that is not homotopic to the boundary. There exists a homeomorphism $\phi_{ij}:\Sigma_i \rightarrow \Sigma_j$ so that $\phi_{ij}(a_i) = a_j$ for all $i,j \in \{1,2,3\}$. 
   
   Let $H$ be a torsion-free one-ended hyperbolic group that does not split over a virtually cyclic subgroup, and let $X_H$ be a finite cell complex with $\pi_1(X_H) \cong H$. Suppose there exists an infinite-order element $h \in H$ represented by a closed curve $a_h$ on $X_H$ so that there exists a degree-2 cover $X_H' \rightarrow X_H$ in which $a_h$ lifts to a single closed curve on $X_H'$. (For a concrete example, let $H \cong \pi_1(S) \rtimes \la t \ra$ be the fundamental group of a closed fibered hyperbolic $3$-manifold with fiber a closed surface $S$, and let $h=t$.) For $i \in \{1,2,3\}$, let $\phi_i:X_{H} \rightarrow X_{Hi}$ be a homeomorphism, and let $a_{hi} = \phi(a_h)$. For $i \in \{1,2,3\}$ let $A_i$ be an annulus. Glue one boundary component of $A_i$ to the curve $a_i$ and the other boundary component of the annulus to the curve $a_{hi}$ by homeomorphisms. Let $X$ be the resulting complex, and let $G$ be the fundamental group of $X$. The JSJ decomposition of $G$ over $2$-ended vertex groups contains three maximal hanging Fuchsian vertex groups, $\pi_1(\Sigma_i)$, and three rigid vertex groups, $\pi_1(X_{Hi})$. 
   
    \begin{claim}
      The group $G$ is not commensurably coHopfian.
    \end{claim}
   
   \begin{proof}
    We first construct a degree-3 cover $X' \rightarrow X$. As shown in Figure~\ref{figure:mixedone}, for $i \in \{1,2,3\}$ there exists a degree-$3$ cover $\Sigma_i' \rightarrow \Sigma_i$ so that $\Sigma_i'$ has one boundary component and genus two and so that the preimage of the curve $a_i$ has two components $a_i'$ and $a_i''$, where $a_i'$ covers $a_i$ by degree one and $a_i''$ covers $a_i$ by degree two. Moreover, there exists a closed curve $\gamma_i$ (shown in blue in Figure~\ref{figure:mixedone}) that separates $\Sigma_i'$ into two subsurfaces; one subsurface has boundary $\gamma_i$ and contains the curve $a_i'$, and the other subsurface has two boundary components and contains the curve $a_i''$. Thus, as in Section~\ref{sec:ex}, the boundary components of $\Sigma_i'$ can be glued together to form a degree three cover of the simple surface amalgam $X_0' \rightarrow X_0$. By assumption on the group $H$, the degree-3 cover of the simple surface amalgam extends to a degree-3 cover of $X$ obtained by taking copies of $X_H$ and copies of the degree-two cover $X_H'$ and attaching them along annuli to lifts of the curves $a_{hi}$ on $X_0'$. See Figure~\ref{figure:mixedthree}.
   
    The degree-4 cover $X'' \rightarrow X$ is constructed in analogy to the construction in Section~\ref{sec:ex}. 
    The space $X''$ contains the space $X'$ as a subspace that induces a $\pi_1$-injective embedding of the fundamental group as an infinite-index subgroup, and for $i \in \{1,2,3\}$ to each of the curves $\gamma_i \subset X'$ defined in the paragraph above a copy of $\Sigma_i \cup X_{Hi}$ is glued along the boundary component of $\Sigma_i$. As above, the space $X''$ forms a degree-4 cover of $X$. 
    Then $X'$ embeds $\pi_1$-injectively in $X''$ such that $\pi_1 (X')$ embeds as an infinite-index subgroup of $\pi_1 (X'')$. So, the group $G = \pi_1(X)$ contains a finite-index subgroup isomorphic to $\pi_1(X')$ and an infinite-index subgroup isomorphic to $\pi_1(X')$. Thus, $G$ is not commensurably coHopfian. 
    \end{proof}
   
       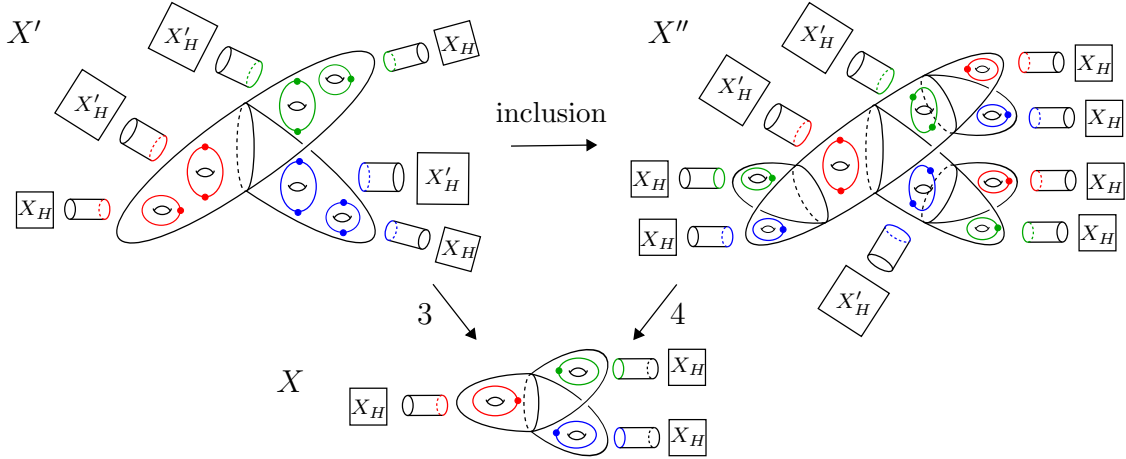
\begin{figure}
	\begin{overpic}[width=.95\textwidth, tics=3 ]{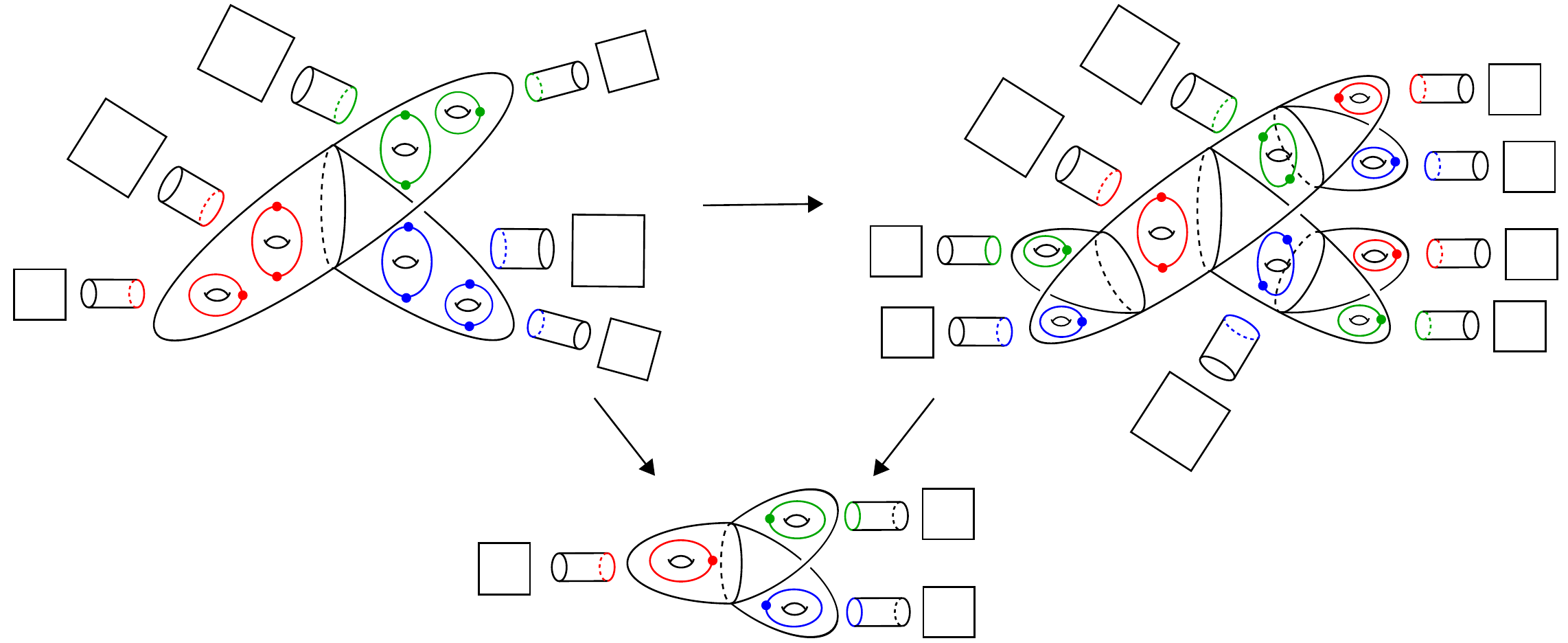} 
	\put(24,6){$X$}
	\put(0,37){$X'$}
	\put(57,37){$X''$}
	\put(43.5,30){\small{inclusion}}
	\put(36.5,12){$3$}
	\put(59,12){$4$}
	\put(30.6,4){\tiny{$X_H$}}
	\put(59,1.5){\tiny{$X_H$}}
	\put(59,7.8){\tiny{$X_H$}}
	\put(1,21.7){\tiny{$X_H$}}
	\put(38.6,36.35){\tiny{$X_H$}}
	\put(38.65,18.2){\tiny{$X_H$}}
	\put(6,31){\tiny{$X_H'$}}
	\put(14,37){\tiny{$X_H'$}}
	\put(37.2,24.5){\tiny{$X_H'$}}	
	\put(56.3,19){\tiny{$X_H$}}
	\put(55.6,24.3){\tiny{$X_H$}}
	\put(95.1,34.6){\tiny{$X_H$}}
	\put(96.05,29.8){\tiny{$X_H$}}
	\put(96.1,24){\tiny{$X_H$}}
	\put(95.4,19.5){\tiny{$X_H$}}
	\put(63.2,32.6){\tiny{$X_H'$}}
	\put(70.5,37.2){\tiny{$X_H'$}}
	\put(73.7,13.5){\tiny{$X_H'$}}
	\end{overpic}
	\caption{\small{The squares represent cell complexes $X_H$ and $2$-fold covers $X_H'$ of $X_H$. The spaces $X'$ and $X''$ are finite covers of the space $X$, and the space $X'$ embeds $\pi_1$-injectively in $X''$, such that $\pi_1 X'$ embeds as an infinite index subgroup of $\pi_1 X''$. Thus, the group $\pi_1(X)$ contains a finite-index subgroup isomorphic to $\pi_1(X')$ and an infinite-index subgroup isomorphic to $\pi_1(X')$.}}
	\label{figure:mixedthree}
    \end{figure}
  \end{example}

   Examples~\ref{ex:sc} and~\ref{prop:rsc} make use of the notion of an \emph{acylindrical submanifold}.
   Let $M$ be a Riemannian manifold and $N \subseteq M$ a locally convex submanifold.
   Let $A$ denote the annulus.
   The submanifold $N$ is said to be \emph{acylindrical} if any $\pi_1$-injective map $(A, \partial A) \rightarrow (M,N)$ is relatively homotopic to a map $(A, \partial A) \rightarrow (N,N)$.
   Equivalently, the subgroup $\pi_1 (N) \leqslant \pi_1 (M)$ is \emph{malnormal} in the sense that $\pi_1 (N) \cap \pi_1(N)^g = \{1 \}$ for all $g \in \pi_1(M) - \pi_1(N)$.
   In particular, if $M$ is a closed hyperbolic manifold and $N$ is a simple closed geodesic, then $N$ is acylindrical in $M$.
   
   We also require the following lemma; we include a proof as we are unaware of a reference.
   
   \begin{lemma} \label{lem:surfaceMapsRelativeToBoundary}
     Let $\Sigma$ and $\Sigma'$ be compact surfaces with boundary with negative Euler characteristic.
     If $f : (\Sigma, \partial \Sigma) \rightarrow (\Sigma', \partial \Sigma')$ is a $\pi_1$-injective map, then $f$ is homotopic to a finite-degree covering map of $\Sigma'$.
   \end{lemma}

   \begin{proof}
    It suffices to prove that $f_* : \pi_1 (\Sigma) \rightarrow \pi_1 (\Sigma')$ embeds $\pi_1 (\Sigma)$ as a finite-index subgroup of $\pi_1(\Sigma')$.
    Let $D\Sigma$ and $D\Sigma'$ be the doubles of $\Sigma$ and $\Sigma'$ along their boundary components.
    The spaces $D\Sigma$ and $D\Sigma'$ are closed surfaces of genus greater than $1$, and the map $f$ determines a $\pi_1$-injective map $F : D\Sigma \rightarrow D\Sigma'$.    
    This last fact follows since $\Sigma$ and $\Sigma'$ have negative Euler characteristic so the subgroups corresponding to the boundary components give malnormal families.
    Thus, the lift $\tilde{f}: \widetilde{\Sigma} \rightarrow \widetilde{\Sigma}'$ sends each boundary component to a unique boundary component.
    If $T$ and $T'$ are the Bass-Serre trees corresponding to the splittings of $\pi_1 (D\Sigma)$ and $\pi_1(D \Sigma')$ given by the doubling, then it follows from the previous observation that the $F_*$-equivariant map $T \rightarrow T'$ induced by the morphism between the graphs of groups is locally injective, and hence globally injective. Thus, $F_*:\pi_1(D\Sigma) \rightarrow \pi_1(D\Sigma')$ is injective.
    
    The fundamental group of a closed surface $D\Sigma$ can only embed in the fundamental group of a closed surface $D\Sigma'$ as a finite-index subgroup, so the map $F$ is homotopic to a finite-degree covering map $F'$, which induces that same map on the fundamental groups.
    The preimage $(F')^{-1}(\Sigma')$ may not be $\Sigma$, but it will be a subspace homeomorphic to $\Sigma$, so we can identify $(F')^{-1}(\Sigma')$ with $\Sigma'$, to obtain a covering map $f': (\Sigma, \partial \Sigma) \rightarrow (\Sigma', \partial \Sigma')$ such that $f_* = f'_*$, so $f'$ is homotopic to $f$.
   \end{proof}
  
   \begin{example} \label{ex:sc} (Commensurably coHopfian.)
    Let $M$ be a closed hyperbolic $3$-manifold, and let $\gamma$ be an embedded locally geodesic closed curve in $M$.  Let $\Sigma$ be a compact surface with positive genus and boundary $\partial \Sigma$ homeomorphic to $S^1$. Identify $\gamma$ with $\partial \Sigma$ via a homeomorphism to obtain a quotient space $X$. The fundamental group $G = \pi_1(X)$ is a one-ended hyperbolic group given by the amalgamation of the $3$-manifold group and the free group $\pi_1(\Sigma)$ along the cyclic groups corresponding to $\gamma$ and $\partial \Sigma$. This amalgamated free product corresponds to the canonical JSJ decomposition for $\pi_1(X)$.

  \begin{prop} \label{claim:ManifoldAndSurface}
   The group $G$ is commensurably coHopfian.
  \end{prop}

  \begin{proof}
   Let $G' \leqslant G$ be a finite-index subgroup and $\varphi: G' \rightarrow G$ an injective homomorphism.
   Without loss of generality, assume that $G'$ is a normal subgroup of $G$ and let $\pi:X' \rightarrow X$ denote the corresponding finite regular cover.
   Take the $\pi$-preimages of $M$ and $\Sigma$ to decompose $X'$ as a collection of homeomorphic $3$-manifolds $M_1', \ldots, M_n'$ and a collection of homeomorphic surfaces with boundary $\Sigma_1', \ldots, \Sigma_m'$ such that $M_i' \rightarrow M$ and $\Sigma_j' \rightarrow \Sigma$ are regular covers.
   
   We first argue that the homomorphism $\varphi: G' \rightarrow G$ is induced by a map $\Phi: X' \rightarrow X$ such that the restriction of $\Phi$ to each $3$-manifold $M_i'$ is a covering map $\Phi_i: M_i' \rightarrow M$.   
   Let $T$ denote the Bass-Serre tree of the JSJ-splitting of $G$. The subgroup $\varphi(G')\leq G$ acts on $T$. 
   Since the group $\pi_1(M_i')$ does not split over a virtually cyclic subgroup, the  subgroup $\varphi(\pi_1(M_i'))$ stabilizes a vertex in $T$. 
   Thus, there exists $g_i \in G$, such that $\varphi(\pi_1(M_i'))^{g_i} \leq \pi_1(M)$. 
   As $\pi_1(M)$ is commensurably coHopfian, $\varphi(\pi_1(M_i'))^{g_i}$ is a finite-index subgroup of $\pi_1(M)$. 
   By Mostow rigidity~\cite{Mostow68}, the covering space corresponding the subgroup $\varphi(\pi_1(M_i'))^{g_i}$ is isometric to $M_i'$, so the homomorphism $\varphi^{g_i}: \pi_1(M_i') \rightarrow \pi_1(M)$ is induced by a covering map $\Phi_i: M_i' \rightarrow M$.  
   Since $X$ and $X'$ have contractible universal covers,  there exists a continuous map $\Psi : X' \rightarrow X$ such that $\Psi_* = \varphi$.
   As $\varphi^{g_i}$ induces the same map on the fundamental group as $\Phi_i$ and the space $X$ is a classifying space, we can homotope $\Psi$ to a map $\Phi$ so that it restricts to $\Phi_i$ on each $M_i' \subseteq X'$ (see~\cite[Proposition 1B.9.]{HatcherAT}).
   Thus, the resulting map $\Phi$ is as specified.
    
    Suppose towards a contradiction that $\Phi:X' \rightarrow X$ is not homotopic to a covering map.
    Let $C_j \subset \Phi^{-1}(\gamma)$ be the set of curves in the full preimage of the amalgamating curve $\gamma \subset X$ that lie in the surface $\Sigma_j'$. 
    After homotopy, we may assume that $C_j$ is a set of disjoint curves and $\partial \Sigma_j' \subseteq C_j$.
    Moreover, since the curves $\gamma \subset M$ and $\partial\Sigma \subset \Sigma$ are acylindrical subspaces, applying a suitable homotopy removes parallel curves in the set $C_j$.
    Lemma~\ref{lem:surfaceMapsRelativeToBoundary} implies that if $C_j = \partial \Sigma_j'$, then $\Phi$ can be homotoped on $\Sigma_j'$, relative to $\partial \Sigma_j'$, to a covering map $\Sigma_j' \rightarrow \Sigma$.
    Since we assumed that $\Phi$ is not homotopic to a covering map we can say, without loss of generality, the set $C_1$ contains a curve that is not a component in $\partial\Sigma_1'$.
    
    Let $\sigma_1, \ldots, \sigma_\ell$ denote the closures of the components of $\Sigma_1' - C_1$.
    Each $\sigma_i$ is mapped by $\Phi$ into either $\Sigma$ or $M$ and we refer to the subsurfaces as either \emph{$\Sigma$-type} or \emph{$M$-type}, accordingly.
    We verify the following claims:
    
    \begin{claim} \label{claim:SigmaType1}
     If $\sigma_i$ is $\Sigma$-type and $\sigma_i \cap \sigma_j \neq \emptyset$, for $j \neq i$ then $\sigma_j$ is $M$-type.
    \end{claim}
    
    \begin{proof}
    Suppose not. Then we would have a $\pi_1$-injective map $f: (\sigma_i \cup \sigma_j, \partial(\sigma_i \cup \sigma_j)) \rightarrow (\Sigma, \partial \Sigma)$.
    By Lemma~\ref{lem:surfaceMapsRelativeToBoundary} the map $f$ is homotopopic to a covering map, which implies that the curves $\sigma_1 \cap \sigma_2$ would have to be homotopic to a boundary curve in $\sigma_1 \cup \sigma_2$, contradicting the fact that $\sigma_i$ and $\sigma_j$ are not annuli.
    \end{proof}
    
    \begin{claim} \label{claim:SigmaType2}
     If $\sigma_i$ has a boundary component in $\partial \Sigma_1'$, then $\sigma_i$ is $\Sigma$-type.
    \end{claim}

    \begin{proof}
    Suppose not. Then we have $\sigma_i$ that intersects $M_j'$ such that $\sigma_i\cup M_j'$ maps $\pi_1$-injectively into $M$.
    But $\sigma_i \cup M_j'$ has cohomological dimension 3 while on the other hand the fundamental group  will be infinite ended so it cannot induce a finite index embedding into $\pi_1 M$.
    By Strebel's theorem~\cite{Strebel77}, such a map cannot exist.
    \end{proof}
    
    Combining Claims~\ref{claim:SigmaType1} and~\ref{claim:SigmaType2} with the assumption that $C_1'$ contains a curve not in $\partial \Sigma_1'$ we can deduce that there exists at least some $\sigma_j$ that is $M$-type.
    Moreover, for each $\Sigma$-type $\sigma_i$, since $\Phi$ induces a $\pi_1$-injective map $(\sigma_i, \partial \sigma_i) \rightarrow (\Sigma, \partial \Sigma)$.
    Lemma~\ref{lem:surfaceMapsRelativeToBoundary} implies we can homotope, relative to its boundary curves, $\sigma_i \rightarrow \Sigma$ to a covering map.

    
    Under the regular covering map $\pi:X' \rightarrow X$ corresponding to the finite-index subgroup $G' \leq G$, each boundary component in $\partial \Sigma_j'$ covers $\partial \Sigma$ with degree $d$ for some $d \in \N$. Thus, the Euler characteristic satisfies $$\chi(\Sigma_j') = d | \partial \Sigma_j' | \cdot \chi( \Sigma).$$  
    
    The degrees of the covering maps from $M_i'$ to $M$ given by either $\pi$ or $\Phi$ must coincide since they are determined by the ratio of the volumes.
    Similarly, if $\gamma'$ is a curve in $\partial \Sigma_i'$, then since $\gamma'$ and $\gamma$ are geodesic curves in some $M_j'$ and $M$ respectively, the degrees of the covering map $\gamma' \rightarrow \gamma$ given by $\pi$ and $\Phi$ must also coincide since it is given by the ratios of their length. It follows that each component in $\partial \Sigma_i'$ covers $\gamma$ with degree $d$.
    

    Thus, we can deduce that 
    
      \begin{align*}
      |\chi(\Sigma_1')| & = \sum_i |\chi(\sigma_i)|  
			 > \sum_{\sigma_i\textrm{ is } \textrm{$\Sigma$-type}} |\chi(\sigma_i)| \\
                       & = \sum_{\sigma_i\textrm{ is } \textrm{$\Sigma$-type}} \deg(\Phi: \partial \sigma_i \rightarrow \gamma)|\chi(\Sigma)|  
                        > d|\partial \Sigma_1'| \cdot |\chi(\Sigma)|.
      \end{align*}
   
   (We let $\deg(\Phi: \partial \sigma_i \rightarrow \gamma$) denote the sum of the degrees of the map restricted to each component in $\partial \sigma_i$.) 
   The first inequality follows from discarding the $M$-type surfaces $\sigma_i$, and the second inequality follows from only counting the degrees of the curves in $\partial \Sigma_1'$. This contradicts the previous equality, and thus, $\Phi$ is homotopic to a covering map. Therefore, $\varphi(G')$ is a finite-index subgroup of $G$. 
  \end{proof}

  \end{example}

  \section{Commensurably coHopfian groups without hanging Fuchsian subgroups in their JSJ decomposition} \label{sec:mixed2}
   In this section we provide an example of a one-ended hyperbolic group with non-trivial JSJ decomposition and only rigid and two-ended vertex groups.
   The key point is that we choose the rigid vertex groups to be commensurably coHopfian.

   \begin{example} \label{prop:rsc}
   Let $M$ and $N$ be closed hyperbolic $3$-manifolds.
   For simplicity we will assume that $\pi_1(M)$ and $\pi_1(N)$ are incommensurable.
   Let $\gamma \subseteq M$ and $\sigma \subseteq N$ be simple closed geodesics, and let $A$ be an annulus.
   Let $X$ be the space obtained from $M \sqcup A \sqcup N$ by gluing one boundary component of the annulus to $\gamma$ and the other to $\sigma$.

  \begin{claim} 
   $G = \pi_1 X$ is commensurably coHopfian.
  \end{claim}
   \begin{proof}
    The proof follows a similar strategy to Claim~\ref{claim:ManifoldAndSurface}.
    Let $G' \leq G$ be a finite-index subgroup and $\varphi: G' \rightarrow G$ is an injective homomorphism.
    Assuming $G'$ is a normal subgroup, let $X' \rightarrow X$ be the finite-sheeted, regular cover corresponding to $G'$.
    Considering the $\varphi$-preimages of $M, N,$ and $A$, decompose $X'$ as a graph of spaces with vertex spaces $M_{u_1}, \ldots, M_{u_m}$ and $N_{v_1}, \ldots, N_{v_n}$ and edge spaces $A_{e_1}, \ldots, A_{e_a}$.
    
    As $\pi_1 (M)$ and $\pi_1 (N)$ do not split over a virtually cyclic group, are commensurably coHopfian, and are incommensurable with each other, there exists $g_i, h_i \in G$ such that $\varphi(\pi_1 (M_{u_i}))^{g_i}$ is a finite-index subgroup of $\pi_1(M)$ and $\varphi(\pi_1 (N_{v_i}))^{h_i}$ is a finite-index subgroup of $\pi_1( N)$. By Mostow rigidity, there exist covering maps $\Phi_{u_i} : M_{u_i} \rightarrow M$ and $\Phi_{v_i} : N_{v_i} \rightarrow N$ that correspond to the embeddings $\varphi^{g_i}: \pi_1(M_{u_i}) \rightarrow \pi_1(M)$ and $\varphi^{h_i}: \pi_1(N_{v_i}) \rightarrow \pi_1(N)$.
    
    The spaces $X'$ and $X$ have contractible universal covers, so there exists a continuous map $\Phi: X' \rightarrow X$ such that $\Phi_* = \varphi$.
    After homotopy, $\Phi$ restricts to $\Phi_{u_i}$ on $M_{u_i}$ and $\Phi_{v_i}$ on $N_{v_i}$.
    The map $\Phi$ may be homotoped to a covering map since any annulus mapping $(A, \partial)$ to either $(M, \gamma)$ or $(N,\sigma)$ can be homotoped into $\gamma$ or $\sigma$, since $\gamma$ and $\sigma$ are acylindrical subspaces of $M$ and $N$.
    Thus, $\varphi(G')$ is a finite-index subgroup of $G$.
   \end{proof}
  \end{example}

\bibliographystyle{alpha}
\bibliography{refs}

\end{document}